\documentclass[letterpaper,11pt]{amsart}
\usepackage{times}
\usepackage[margin=1.20in]{geometry}
\usepackage{graphicx}
\usepackage{amssymb}
\theoremstyle{plain}
\newtheorem{theorem}{Theorem}%[section]
\newtheorem{lemma}[theorem]{Lemma}

\newtheorem{prop}[theorem]{Proposition}
\theoremstyle{definition}

\theoremstyle{remark}

\newcommand{\suppress}[1]{}
\newcommand{\zset}{\mathbb Z}

\newcommand{\rset}{\mathbb R}
\newcommand{\cset}{\mathbb C}

\numberwithin{equation}{section}
\newcommand{\ep}{{\varepsilon}}

\newcommand{\cL}{{\mathcal L}}

\newcommand{\be}{\begin{equation}}
\newcommand{\ee}{\end{equation}}
\newcommand{\bea}{\begin{eqnarray}}
\newcommand{\eea}{\end{eqnarray}}
\newcommand{\bean}{\begin{eqnarray*}}
\newcommand{\eean}{\end{eqnarray*}}

\DeclareMathOperator{\Var}{Var}

\newcommand{\valnorm}[3]{\| #1(#2^n) \|_{#3}}
\newcommand{\coeffnorm}[2]{\| #1 \|_{#2}}
\newcommand{\pmo}{(\pm 1)}
\usepackage{color}

\title{
 Quasi-random multilinear polynomials
}
\author{Gil Kalai}\address{Hebrew University} \email{gil.kalai@gmail.com} \author{Leonard J. Schulman}\address{Caltech, Engineering and Applied Science MC305-16, Pasadena CA 91125, USA} \email{schulman@caltech.edu}

\begin{document}

\begin{abstract}
We consider multilinear \textit{Littlewood polynomials}, 
polynomials in $n$ variables in which a specified set of monomials $U$ have $\pm 1$ coefficients, and all other coefficients are $0$. We provide upper and lower bounds (which are close for $U$ of degree below $\log n$) on the minimum, over polynomials $h$ consistent with $U$, of the maximum of $|h|$ over $\pm 1$ assignments to the variables. (This is a variant of a question posed by Erd\"os regarding the maximum on the unit disk of univariate polynomials of given degree with unit coefficients.) We outline connections to the theory of quasi-random graphs and hypergraphs, and to statistical mechanics models. Our methods rely on the analysis of the Gale-Berlekamp game; on the constructive side of the generic chaining method; on a Khintchine-type inequality for polynomials of degree greater than $1$; and on Bernstein's approximation theory inequality.
\end{abstract}

\clearpage

\maketitle
\pagenumbering{arabic}
\section{Introduction}
In this article a \textit{Littlewood polynomial} $h$ is a 
polynomial in $n$ variables in which all nonzero monomial coefficients are units. Such $h$ has the expansion $h(x_1,\ldots,x_n)=\sum_{S}h_S x^S$, where $S$ is a nonempty subset of $[n]$, $x^S=\prod_{i\in S}x_i$, and $h_S \in C$, where $C$
is the unit circle in $\cset$. If all nonzero $h_S$ are in $ \pm 1$ we say the polynomial is real.
Given a set $U$ of monomials in $n$ variables, always excluding the empty monomial (constant function), the real (or complex)
\textit{Littlewood family} of polynomials 
$\cL_{U,\rset}$  (or $\cL_{U,\cset}$)
is the set of real (or complex) Littlewood polynomials whose set of nonzero coefficients equals $U$. We say $U$ is multilinear if every monomial in it is; that $U$ is homogeneous if all monomials have the same total degree; and that $U$ is $d$-bounded $(d 
\geq 1)$ if the maximum total degree of a monomial is $d$. We also think of $U$ as a hypergraph, and a polynomial $h$ as a real or complex-valued function on the edges.

Let $D$ be the unit disk in $\cset$ (so $C=\partial D$), and $I$ the closed interval $[-1,1]$. 
We abbreviate the sup norm of $h$ on a set $X$ as $\| h(X) \|_{\infty}$.
This work is
primarily concerned with bounding
\[ \inf_{h \in \cL_{U,\rset}} \valnorm{h}{I}{\infty} \]
as a function of $U$.
 Observe that for multilinear $U$ and any $h \in  \cL_{U,\cset}$, 
 $\valnorm{h}{\pmo}{\infty} = \valnorm{h}{I}{\infty}$.
Let $u=|U|$ and for $i \in [n]$ let $u_i=|\{S: i \in S \in U\}|$. It turns out that for small $d$, a key role is played by the quantity $\sum \sqrt{u_i}$.
 Our main result is: 
\begin{theorem} \label{main-thm} For $d$ fixed and $U$ a $d$-bounded multilinear
Littlewood family,
\be \inf_{h \in \cL_{U,\rset}} \valnorm{h}{I}{\infty} \in 
\Theta_d\left( \sum \sqrt{u_i} \right) 
\ee
More precisely and without assumptions on $d$,
\be \max\{ \frac{1}{3^{d-1}2^{5/2} (d-1)^{3/2}d}
\sum \sqrt{u_i}
,
\sqrt{u}
\}
\leq 
\inf_{h \in \cL_{U,\rset}} \valnorm{h}{I}{\infty}  \leq \min \{ 
3.34 \sum \sqrt{u_i},
\sqrt{2(n+1)u}
 \} \label{main-bd} \ee
\end{theorem}

The lower bound $\sqrt{u}$ is straightforward; for large $d$ it is actually stronger than the bound controlled by $\sum \sqrt{u_i}$ because of the exponential loss in $d$, but
for small $d$ it is weak. We discuss this further below.
The upper bound $\sqrt{2(n+1)u}$ is also straightforward; generally it is weaker than the bound controlled by $\sum \sqrt{u_i}$, sometimes substantially so (e.g., 
 consider the disjoint union of a clique on $n^{\ep}$ vertices, $\ep>1/2$, with a matching on $n-n^{\ep}$ vertices). There are situations where it is stronger, especially when $d$ is large and the $u_i$ are balanced, but it cannot be stronger by more than a factor of $3.34 \sqrt{d/2}$, due to Cauchy-Schwarz: $\sum \sqrt{u_i} \leq \sqrt{n \sum u_i} \leq \sqrt{dnu}$.
 
It is not really necessary in the theorem to assume $U$ is multilinear, since exponents may be reduced by $2$ without affecting evaluation on $\pmo^n$; in the process however monomials identify and their coefficients add. One can write down bounds in terms of general coefficients but for brevity and because of connections we sketch shortly, we restrict our statements to Littlewood polynomials.

It is interesting to also study 
\[ \inf_{h \in \cL_{U,\cset}} \valnorm{h}{D}{\infty}. \]
 By the maximum modulus principle, for any $U$ (not necessarily multilinear) and $h \in  \cL_{U,\cset}$,
 $\valnorm{h}{C}{\infty} = \valnorm{h}{D}{\infty}$. Our upper and lower bounds can be replicated with mild loss for $\inf_{h \in \cL_{U,\cset}} \valnorm{h}{D}{\infty}$; we omit doing so since no new ideas are involved. The upper bound does not require the multilinear assumption.

\section{Background and related literature}
\subsection{An elementary lower bound}
For a polynomial $h=\sum_S h_S x^S$ (not necessarily multilinear, so $S$ may range over multisets), and for $1\leq \alpha< \infty$,
let \[\valnorm{h}{\pmo}{\alpha}= \left(2^{-n} \sum_{x \in \pmo^n} |h(x)|^\alpha\right)^{1/\alpha}\] and 
 \[ \valnorm{h}{C}{\alpha} = \left( (2\pi)^{-n} \int_{0}^{2\pi} \cdots  \int_{0}^{2\pi} |h(e^{i\theta_1},\ldots,e^{i\theta_n})|^\alpha  d\theta_1\cdots d\theta_n\right)^{1/\alpha}. \]
These are the \textit{evaluation norms} of the polynomial (normalized to be means), with $\alpha=\infty$ being the max norm. The \textit{coefficient norms} are $\coeffnorm{h}{\alpha}=\left(\sum_S |h_S|^\alpha\right)^{1/\alpha}$ with of course $\coeffnorm{h}{\infty}=\max_S |h_S|$ (these are not normalized to be means). 
The power-means inequality gives 
\be \valnorm{h}{C}{\infty} \geq  \valnorm{h}{C}{2} \label{elem-lb1} \ee
and considering $C$ as the circle group, we have the Plancherel identity
\be \valnorm{h}{C}{2} 
= \coeffnorm{h}{2} =u^{1/2}
\label{elem-lb2} \ee
Let $L(d)$ denote
the complex univariate polynomials of degree $d$, all of whose coefficients are units; these are customarily called Littlewood polynomials (we have adopted the term more broadly). In 1957
Erd\"os~\cite{Erdos57} asked for 
the minimum over $L(d)$ of the sup norm of polynomials on the unit disk $D$, i.e., for
$\min_{h \in L(d)} \max_{x \in D} |h(x)|$;  he also investigated the analogous question for trigonometric polynomials~\cite{Erdos62}. 
Littlewood~\cite{Littlewood66,Littlewood68} asked for 
$\min_{h \in L(d)} ((\max_{x \in C}  |h(x)|)-(\min_{x \in C}  |h(x)|))$; this is known as Littlewood's flatness problem. The questions have been investigated for both complex and real (i.e., $\pm 1$) coefficients.
Our paper concerns Erd\"os's question as applied to multilinear polynomials. 

For univariate Littlewood polynomials of degree $d$,~\eqref{elem-lb1}-\eqref{elem-lb2} specializes to
$\|h(C)\|_\infty \geq  \|h(C)\|_2 
= \coeffnorm{h}{2} \geq
 \frac{1}{\sqrt{d+1}} \coeffnorm{h}{1}. $
The last two terms are necessarily equal, and equal to $\sqrt{d+1}$. There has been careful study of 
how tight this inequality is.
First, for the case of real (i.e., $\pm 1$) coefficients, the Rudin-Shapiro polynomials~\cite{Shapiro51,Rudin59,Littlewood66} match the lower bound to within a factor of $\sqrt{2}$. It is an open problem to close this gap. For the case of complex coefficients, the gap to the lower bound is $(1+\ep_n)$ for some $\ep_n \to 0$ (these are called ``ultraflat'' polynomials); this construction is due to Kahane~\cite{Kahane80} (and see a correction in~\cite{QueffelecS96}). For a survey of these and considerably more results in the univariate case, see~\cite{Erdelyi01polynomialswith,Borwein02}; for some recent results see~\cite{JedwabKS13,Schmidt14}.

A bound analogous to~\eqref{elem-lb1}-\eqref{elem-lb2} holds with the group $\zset/2$ replacing the circle group: Let $U$ be multilinear.
For $h\in \cL_{U,\rset}$, \be \valnorm{h}{\pmo}{\infty}\geq \valnorm{h}{\pmo}{2} = \coeffnorm{h}{2}=u^{1/2}. \label{elem-lb3} \ee
(This establishes the simpler of the two lower bounds in~\eqref{main-bd}.)
The inequality is again power means. 
The equality can be understood in two 
equivalent ways. The first is that each multilinear monomial $x^S$ is a multiplicative character on the group $(\zset/2)^n$ and so the transformation between the vector of coefficients $h_S$ and the vector of evaluations $h(\pmo^n)$ is the Fourier transform over  $(\zset/2)^n$, so this is the Parseval identity. The second way is to consider each $x_i$ to be an iid uniform random variable on $\pm 1$. Then the monomials $x^S$ are also uniform on $\pm 1$, and, since they are multilinear, they are pairwise independent. Since $E(h(x))=0$, $\valnorm{h}{\pmo}{2}^2 =E(h^2(x))= \Var(h(x))=\sum_S \Var(h_S x^S)=\sum_S h^2_S= \coeffnorm{h}{2}^2$. 

\subsection{Quasi-random graphs and hypergraphs}
Chung, Graham and Wilson~\cite{ChungGW89} (and for precursors see~
\cite{Thomason87a,Thomason87b,Rodl86,FranklRW86}) define a family of graphs $G_i=(V_i,E_i)$ to be quasi-random if, among several equivalent characterizations, for any functions $x_i:V_i\to \pm 1$, \[\left|\sum_{\{u,v\} \in E_i} x_i(u)x_i(v) - \sum_{\{u,v\} \notin E_i} x_i(u)x_i(v)\right| \in o(|V_i|^2). \] 
Simply by setting $h_{uv}=1$ for edges and $h_{uv}=-1$ for non-edges, this expression exactly corresponds to evaluation of a multilinear Littlewood polynomial with $U=\binom{V_i}{2}$. A quasi-random graph is one whose corresponding polynomial achieves a non-trivial, that is sub-quadratic (though not necessarily optimal in the exponent) sup-norm bound. 

Let $A_i$ be the adjacency matrix of $G_i$ (and note the graph has no loops or multiple edges); and consider $x$ as a column vector (thus $\|x\|_2=\sqrt{|V_i|}$). Then an equivalent characterization of quasi-random graphs is that
\[ \left| x^* (2A_i+I-J) x \right| / \|x\|_2^2 \in o(|V_i|) \]
with $I$ the identity matrix and $J$ the all-ones matrix.

In fact, an example given in~\cite{ChungGW89} yields a
\textit{deterministically constructible} polynomial of degree $2$ with an optimal (up to the constant) norm bound. Specifically, the Paley graph $P_n$ (vertices are elements of the field $GF(n)$ for $n=p^\alpha$, $p$  prime $ \equiv 1 \bmod 4$; $\{i, j\}$ is an edge if $i-j$ is a square) gives a homogeneous degree $2$ polynomial $h_{P_n}$ with $u=\binom{n}{2}$ such that\footnote{The Paley graph is strongly regular~\cite{KrivelevichS06} and its adjacency matrix has the spectrum: $\lambda_1=(n-1)/2$ with multiplicity one (eigenvector $v_1=$ the constant function), and $\lambda_2,\ldots,\lambda_n \in (-1\pm \sqrt{n})/2$ (each with multiplicity $(n-1)/2$). Then with $A$ the adjacency matrix of $P_n$, $(2A+I-J)$ commutes with $A$ and has spectrum $0$ for $v_1$, $2\lambda_i+1$ for remaining eigenvalues, so by Courant-Fischer, for $x \in (\pm 1)^n$, $|x^* (2A+I-J) x| \leq \|x\|_2^2 \max_{i\geq 2} |2\lambda_i+1|=n^{3/2}$.
} 
\[ \|h_{P_n}((\pm 1)^n)\|_{\infty} \leq n^{3/2}. \]
Theorem~\ref{main-thm} establishes on the other hand that $ \|h_{P_n}((\pm 1)^n)\|_{\infty} \in \Omega(n^{3/2})$.

This connection between graphs and polynomials is the reason we adopt the term ``quasi-random'' for a polynomial which achieves a sup-norm which is $o(u)$. However this is a fairly soft requirement, and as indicated in Theorem~\ref{main-thm}, a random polynomial achieves a much tighter bound.

We mention also that Simonovits and S{\'o}s~\cite{SimSos91} established another characterization of quasi-random graphs: they are those which, in the partitions provided by the Szemer{\'e}di regularity lemma, have almost all inter-block edge densities close to $1/2$. Through this connection quasi-random graphs play a role in the theory of graph limits; see Lov{\'a}sz's monograph~\cite{lovaszGraphLimits}.

Similar quasi-randomness questions have also been explored for hypergraphs~\cite{havilandT89,ChungG89,chung90,kohayakawaRodl02,ConlonHPS12}. The last of these defines property DISC$_r(\delta)$ for a $d$-uniform hypergraph to be that as $n \to \infty$ and for all $S \subseteq [n]$, the number of induced $d$-hyperedges in $S$ is in the interval $r \binom{|S|}{d} \pm \delta n^d$; a hypergraph achieving small $\delta$ (i.e., having a number of induced hyperedges that is closely predicted just by $|S|$)
is to be considered quasi-random. We can again convert a $d$-uniform hypergraph to a multilinear polynomial by putting, for all $S \in \binom{[n]}{d}$, $h_S=1$ if $S$ is in the hypergraph and $h_S=-1$ otherwise. However, for $d>2$, 
there does not seem to be a close connection between being quasi-random in the sense of DISC and being quasi-random in the sense of the sup-norm of the polynomial. DISC essentially counts how many hyperedges intersect but are not supported on $S$; it does not distinguish further among these hyperedges based on the cardinality of their intersection with $S$.
 
However, a different single-number summary of the intersection frequencies starts with the perspective 
that graph expansion measures the quality of the graph as a cut approximator of the complete graph, which is the same thing as how well it approximates the second \textit{binary Krawtchouk polynomial}~\cite{kraw29,leven95} at $1/2$; equivalently, edges are counted according to the parity of their intersection with $S$.
The polynomial sup-norm that we study generalizes this for $U=\binom{[n]}{d}$: it measures 
the quality of a $d$-uniform hypergraph as a cut approximator of the $d$'th binary Krawtchouk polynomial at $1/2$. Equivalently, hyperedges are counted according to the parity of their intersection with $S$.
 This same quantity is the key measure of hypergraph expansion in the work of Ta-Shma on small $\varepsilon$-biased sets~\cite{TaShma17}. 

\subsection{Spin glasses} In statistical mechanics and condensed matter, physicists have long studied ``local Hamiltonians'', that is, energy functions defined in terms of local interactions between particles. Locality usually takes its meaning from a graph embeddable without distortion in one to three dimensions of space (often but not always a lattice). The Ising model (see, e.g.,~\cite{PathriaB11}) is a common model in which each particle takes on just two states $\pm 1$, and, with $V$ denoting the set of particles and $U$ the edges between them, the energy of the system in state $x:V\to \pm1$ is $H(x)=-\sum_{\{u,v\} \in U} h_{\{u,v\}} x(u)x(v)$. A ferromagnetic Ising model is one in which $h_{\{u,v\}}\geq 0$ for all $\{u,v\}\in U$; this is a model whose ground state is frustration-free, namely, all contributions to the energy are simultaneously at their least possible value. In our setting this ``ferromagnetic polynomial'' $h$ achieves the worst-possible (largest) norm bound. On the other hand \textit{spin glasses} are models $h$ which 
necessarily exhibit a great deal of cancelation among the contributions to the energy; in our setting this is a polynomial with a strong norm bound. Typically in physics one considers a random $h$ to be a spin glass. Our work shows that such systems can achieve non-trivially low energies, i.e., significantly less ``frustration'' than one would see in a random state $x$.

\subsection{Bent functions}
Bent functions~\cite{Rothaus76,dillon72,dillon74} are functions
$\xi:\{0,1\}^n\to\{0,1\}$ such that the function $(-1)^\xi$, which is of course perfectly ``flat'' (all values have a common norm), has a Fourier transform 
 over
$(\zset/2)^n$ that is also perfectly flat, namely 
(with
the transform scaled by $2^{-n/2}$) the transform is $(-1)^\eta$ for
another function $\eta:\{0,1\}^n\to\{0,1\}$. Thus each of $(-1)^\xi$ and $(-1)^\eta$ is a perfectly
flat multilinear Littlewood function; as we have seen, this
is possible only due to their high degree. Bent functions
have been extensively studied particularly in view of applications in
cryptography~\cite{Tokareva15,carletM16} and quantum
computation~\cite{childsKOR13}.

\section{Proof of Theorem~\ref{main-thm}: upper bound}
The case $d=1$ is trivial; now assume $d\geq 2$.
We employ the generic chaining following Talagrand~\cite{Talagrand96}.
Order the variables so that $u_1 \geq \ldots \geq u_n\geq 1$. Set $K= \lceil
\lg (n+1) \rceil$ and partition $U$ into the sets $V_k$, $1\leq k \leq
K $, $V_k=\{S\in U: S \cap [2^{k-1} ,2^k-1] \neq \emptyset \text{ and
} S \cap [2^k,n]=\emptyset \}$. Then $v_k:=|V_k|\leq \sum_{i= 2^{k-1}
}^{\min\{2^k-1,n\}} u_i$. Note that $v_1\leq 1$. 

\begin{lemma} 
There is an $h$ s.t.\ 
$ \max_{x \in \pmo^n} |h(x)| \leq 1.18 \sum_{k\geq 1} 2^{k/2} v_k^{1/2}$.
 \end{lemma}
\begin{proof} 
For $x \in \pmo^n$ let $x(k)=(x_1,\ldots,x_{2^k-1},0,\ldots,0)$ (with
$x(0)=(0,\ldots ,0)$ and $x(K)=x$). It is useful to expand $h(x) =
\sum_{k=1}^{K} h(x(k))-h(x(k-1))$. For $k=1,\ldots,K$, in order, we will fix the coefficients $\{h_S: S \in V_k\}$ so that for all $x$,  $|h(x(k))-h(x(k-1))| \leq 1.18 \cdot 2^{k/2} v_k^{1/2}$.

Suppose that $h_S$ has been fixed for all $S \in V_1 \cup \ldots \cup V_{k-1}$. 
Choosing $h_S$ uniform iid for all $S \in V_k$, $h(x(k))-h(x(k-1))$ is a symmetric random walk of length $v_k$. 
By a standard Chernoff bound~\cite{Meyer72}, for any $\lambda>0$, $\Pr [|h(x(k))-h(x(k-1))| > \lambda 2^{k/2} v_k^{1/2}] \leq 2\exp(-\lambda^2 2^{k-1})$. 

Taking a union bound over $x(k)$, we have \[\Pr[ \exists x(k): |h(x(k))-h(x(k-1))| > \lambda 2^{k/2} v_k^{1/2}] \leq 2^{2^k}\exp(-\lambda^2 2^{k-1})
= 2^{(1- \lambda^2(\lg e)/2)2^k}
.\]
In order that there exist a satisfactory choice of coefficients $h_S$ for $S \in V_k$ it suffices that $\lambda>\sqrt{2/\lg e}$; e.g., we may take $\lambda=1.18$.
\end{proof}

To apply this in the theorem, consider any $1<k\leq K$. Then 
\bean 2^{k/2}v_k^{1/2} &\leq & 2^{k/2} \left(\sum_{i=2^{k-1}}^{\min(2^k-1,n)} u_i\right)^{1/2}
\\ &\leq & 2^{k/2} (2^{k-1} u_{2^{k-1}})^{1/2} \\ &\leq & 2^{k-1/2} u_{2^{k-1}-1}^{1/2}
\\ &\leq & 2^{3/2} \sum_{i=2^{k-2}}^{2^{k-1}-1} u_i^{1/2}
. \eean 
Also, for $k=1$ observe that $2^{k/2} v_k^{1/2} \leq \sqrt{2} \cdot 1 \leq \sqrt{2} u_{2^{K-1}} \leq 2^{3/2} u_{2^{K-1}}
$.

Thus 
$1.18 \sum_{k=1}^K 2^{k/2} v_k^{1/2} 
\leq 1.18 \cdot 2^{3/2} \sum_{i=1}^{2^{K-1}} u_i^{1/2} \leq 3.34 \sum_{i=1}^n u_i^{1/2}$. 

\section{Proof of Theorem~\ref{main-thm}: lower bound}
\label{lb-sec}
The simple bound of $u^{1/2}$ was shown in~\eqref{elem-lb3}. The case $d=1$ is trivial; now assume $d\geq 2$. In a high-level view, the proof strategy is to reduce to the linear case, in a manner modeled after the solution of the 
Gale-Berlekamp game~\cite{BrownS71,GordonW72,FishburnS89}, which concerns the case of a ``bipartite" quadratic polynomial. In our situation, several additional steps are needed: one being to find a suitable partition of the variables that replaces this bipartite structure, a second involving a generalization of the Khintchine-Kahane inequality, and the last being to use a bound from approximation theory to overcome cancelation by nonlinear terms (this will be clearer below). 

Our first step is to split the variables
 into two sets $X$ and $Y$, with a desirable property on the
degrees (in the hypergraph of monomials) of the variables in $X$. 

Consider selecting a set $X$ as follows:  each $i$ is
included in $X$ independently with probability $\frac{1}{2(d-1)}$. Let
$M_i=\{a \in U: a \cap X = \{i\}\}$, and  
 $m_i = |M_i|$. Note that $m_i=0$ for $i \notin X$.
 
 Conditioned on $i \in X\cap a$,
 $\Pr(a \in M_i)\geq 1-\frac{|a|-1}{2(d-1)}$ (applying independence and a union bound).  Thus (without conditioning), for any $i$, $E m_i \geq \frac{u_i}{4(d-1)}$. 

Since $m_i \leq u_i$, we can write $\frac{u_i}{4(d-1)} \leq (1-\Pr(m_i > \frac{u_i}{8(d-1)})) \frac{u_i}{8(d-1)}  + \Pr(m_i > \frac{u_i}{8(d-1)})u_i$, 
consequently $\Pr(m_i > \frac{u_i}{8(d-1)})
\geq \frac{1}{8(d-1)}$. So for any $i$, $E m_i^{1/2} \geq 2^{-5/2} (d-1)^{-3/2} u_i^{1/2}$. 

Consequently
\be E \sum m_i^{1/2} \geq 
\frac{1}{2^{5/2}  (d-1)^{3/2}} \sum u_i^{1/2}.
\label{goodX} \ee 

Fix a set $X$ achieving~\eqref{goodX}. We proceed to show the existence of an assignment to the variables achieving large $|h(x)|$.

Let $Y=[n]-X$
and write $h=\sum_{s=0}^d h_s$ where 
\[ h_s=\sum_{S\subseteq X, T \subseteq Y, S \cup T \in U, |S|=s} h_{S,T} x^S y^T\]
and where $x^S=\prod_{i\in S} x_i$, etc.

Decompose $h_1$ into $h_1=\sum_{i\in X} x_i h_{i*}$ where \[ h_{i*}=\sum_{T \subseteq Y: T \cup \{i\} \in U} h_{\{i\},T}y^T.\] 
$h_{i*}$ is degree $(d-1)$-bounded, and $\coeffnorm{h_{i*}}{1}=m_i$. Now
choose each $y_i \in \pm 1$ randomly, uniformly and independently. We show
\be E |h_{i*}|  \geq 3^{1-d} m_i^{1/2}. \label{hi*} \ee
as an instance of the following Khintchine-type bound~\cite{khintchine23,kahane64,latalaO94}:

\begin{prop}
Let $p$ be a
multilinear Littlewood polynomial with $w$ monomials of degrees $\leq \delta$,
 in variables $\{y_j\}$ which are iid uniform in $\pm 1$. Then $E|p| \geq 3^{-\delta}(Ep^2)^{1/2}= 3^{-\delta}w^{1/2} $. 
 
Conversely there is a $p$ s.t.\ 
 $E|p| \leq
2^{(1-\delta)/2}w^{1/2}$. 
\label{hi*prop} \end{prop} 
To obtain~\eqref{hi*} apply this upper bound to the polynomial $p=h_{i*}$, with coefficients $p_T=h_{\{i\},T_1}$, on the variables in $Y$.

\proof
Note
\[  E p^2 = \sum_{T_1,T_2} p_{T_1}p_{T_2} E y^{T_1+T_2} 
 = \sum_{T_1,T_2: T_1\oplus T_2=\bar{0}} p_{T_1}p_{T_2}
 = \sum_{T} p_{T}^2 =w \]
where $+$ in the exponent is multiset addition (i.e., repetitions are retained), 
while $\oplus$ is multiset addition modulo $2$ (and $\bar{0}$ denotes multisets with all repetition numbers even). This is because any monomial with an odd repetition number contributes $0$ to the expectation.

Next we require an upper bound on $E p^4$; $w^3$ is obvious but not useful. However, application of the Bonami Lemma~\cite{Bonami70,ODonnell14} gives that
$(E p^4)^{1/4} \leq 
 3^{\delta/2} (E p^2)^{1/2}$, and consequently 
\be E p^4 \leq 9^{\delta} (E p^2)^2 = 9^{\delta} w^2. \label{4th-moment} \ee

At this point, for a weaker bound, we could simply apply the Paley-Zygmund inequality~\cite{PaleyZ32}, which in this setting provides $\Pr(|p|> \frac12 (Ep^2)^{1/2})\geq \frac{9}{16} (Ep^2)^2/(Ep^4) \geq \frac{9}{16} 9^{-\delta}$. Thus $E|p|> \frac{1}{32} 9^{1-\delta} w^{1/2}$. 

Instead, to complete the proof of Prop.~\ref{hi*prop}
we rely on the idea of Berger~\cite{Berger97} to use the inequality:
\[ \forall x,a>0: \quad x \geq \frac{3^{3/2}}{2a}(x^2-x^4/a^2) \]
Applying this inequality with $x=|p|$, and using~\eqref{4th-moment}, we find
\bean E |p| & \geq & \frac{3^{3/2}}{2a}\left( E(p^2-p^4/a^2)\right) \\
&\geq & \frac{3^{3/2}}{2a}\left(w - w^2 9^{\delta} /a^2 \right)
\eean
Setting $a= 3^{\delta+1/2} w^{1/2}$ we have 
\bean E |p| & \geq & 3^{-\delta} w^{1/2}
\eean
as desired.

For the converse, consider the following polynomial (suggested from an example in~\cite{mekaNV16}) in 
variables $x_{i}$ and $y_{i}$ ($1\leq i \leq \delta$):
\[ p=\prod_{i=1}^{\delta} (1+x_{i}) - \prod_{i=1}^{\delta} (1+y_{i}) \]
This is a constant-free Littlewood polynomial of degree $\delta$ with $w=2^{\delta+1}-2$
monomials.
The only values $p$ can achieve are $0, 2^\delta$ and $-2^\delta$, the latter values occurring with probabilities $2^{-\delta}(1-2^{-\delta})$. So $E|p|=2-2^{1-\delta}$, and consequently $E|p| \leq
2^{(1-\delta)/2}w^{1/2}$. 
\qed

It is an interesting question to resolve the gap between $3$ and $\sqrt{2}$ in the base of this Khintchine bound. 

Now we prove the lower bound of Theorem~\ref{main-thm}. Apply~\eqref{hi*} 
and linearity of expectation to pick the $y_i$'s so that $\sum  |h_{i*}|  \geq 3^{1-d} \sum m_i^{1/2}$; apply~\eqref{goodX} to conclude that
$\sum_i |h_{i*}|  \geq 
3^{1-d} 2^{-5/2} (d-1)^{-3/2} \sum u_i^{1/2}$. 
Then pick the $x_i$'s so that $h_1=\sum_i |h_{i*}|$. 

Having set the $y_i$'s randomly and tailored the $x_i$'s to control $h_1$, we have lost all control over the values of $h_0,h_2,\ldots,h_d$, each of which lies in the range $\pm u$. In particular these terms can easily cancel out the contribution
from $h_1$.

We now show how to remedy the
 potential cancelation of the $h_1$ term by the remaining terms. We do this by changing the assignment to the $x_i$'s.
 (This step of the proof is somewhat unusual, but 
 it is not completely new: we are aware that the same idea was used for a different purpose in~\cite{DinurFKO07}.)
The modification of the $x_i$'s is probabilistic. We use approximation theory to show that there exists a value of $0\leq p \leq 1$ s.t.\ if we flip each $x_i$ independently with probability $p$, then $E|h|$ is large.

Set $-1 \leq z=1-2p \leq 1$. Let $x'_i$ ($1 \leq i \leq n_1$) be independent random variables with $\Pr(x'_i=1)=1-p$, $\Pr(x'_i=-1)=p$. Write $xx'$ for the list of products $(x_1x'_1,\ldots,x_{n_1}x'_{n_1})$.
 Let $f(z)$ be the polynomial
\[f(z)= E_{x'} h(xx',y) = \sum_{s=0}^d (1-2p)^s h_s(x,y) = \sum_{s=0}^d z^s h_s(x,y) \]
What we know about this polynomial is only
 that the linear coefficient of $z$, $h_1(x,y)$,
is large, specifically $h_1(x,y) \geq 
3^{1-d} 2^{-5/2} (d-1)^{-3/2} \sum u_i^{1/2}$. 
It is useful to note that $h_1(x,y)=f'(0)$.
 As noted, we have no control over the coefficients of the remaining terms in $f$. But the lower bound on $f'(0)$ is enough information to guarantee that the polynomial is large somewhere in $[-1,1]$. This is a consequence of a classic result in uniform approximation theory (of which we quote only a special case):

\begin{lemma}[Bernstein~\cite{Natanson61}] Let $f$ be a polynomial of degree $d$ with complex coefficients. Then $\sup_{z \in [-1,1]} |f(z)| \geq |f'(0)|/d$. \end{lemma}

As a result we immediately obtain that there is a setting of the $x_i$'s such that
\[ |h(x,y)| \geq 
 3^{1-d} 2^{-5/2} (d-1)^{-3/2}d^{-1} \sum u_i^{1/2}.\]

\section{Discussion}
\subsection{How quasi-random is the determinant?}
Let $U$ be the collection of monomials that occur as permutation matrices in a $d \times d$ matrix. Here $n=d^2$ and $u=d!$. 

The determinant is in $\cL_{U,\rset}$ and
the upper bound for the determinant is the volume bound of Hadamard: since each row of the matrix has $L_2$ norm at most $\sqrt{d}$, the volume is at most $d^{d/2} = e^{\frac{1}{2} d \log d}$. This qualifies the determinant as quasi-random by our soft definition, since the trivial upper bound is $d! \in e^{d \log d - d + O(\log d)}$.
 More significantly for our current purpose, the $d^{d/2}$ bound is tight for infinitely many $d$ (powers of $2$ and also some other values).

However, using the (easy) upper bound of Theorem~\ref{main-thm}, we have that there exists a Littlewood polynomial $h$ with the same monomials such that $\max_{x \in \pmo^n} |h(x)| \leq
\sqrt{2 (d^2+1) d!}
 \leq e^{\frac12 (d \log d - d) + O(\log d)}$. 
This improves on the determinant by a factor of $\cong e^{d/2}$. (The simple lower bound for this $U$ is $\sqrt{u} = \sqrt{d!}$. The upper and lower bounds are within a factor of merely
$\sqrt{2(d^2+1)}$.)

In short the determinant is fairly good as a quasi-random polynomial, but not competitive with a random polynomial supported on the same monomials.

We remark that matrices with $\pm 1$ entries are a subject of extensive study in mathematical physics and beyond, starting with Wigner~\cite{Wigner55}; however this literature is apparently not too connected to our question, as the focus is on random (symmetric) matrices (not extremal ones) and the entire spectrum, not just the determinant, is of essence.

\subsection{ $ \inf_{h \in \cL_{U,\rset}} \valnorm{h}{D}{\infty} $: complex derandomization}
For any Littlewood family $U$, our lower bound on $ \inf_{h \in \cL_{U,\rset}} \valnorm{h}{I}{\infty} $ is of course also a lower bound on  $ \inf_{h \in \cL_{U,\rset}} \valnorm{h}{D}{\infty} $. Our approximation theory argument can be derandomized in a slightly unusual way.
That is, applying Bernstein's lemma and the maximum modulus principle, we conclude that there is a $z \in C$ such that $h(zx_1,\ldots,zx_{n_1},y_1,\ldots,y_{n_2})$ achieves the lower bound.
Of course we do not know this number $z$ explicitly (any more than 
 in the argument of Section~\ref{lb-sec} we know $p$ and all the new values of the $x_i$'s).
 
\subsection{Large $d$; Khintchine; Anti-concentration} \label{disc:khintchine}
 Our lower bound in Theorem~\ref{main-thm} trivializes at $d$ logarithmic in the size of the hypergraph. 
The exponential loss occurs in the Khintchine-type inequality, Prop.~\ref{hi*prop}. As noted in the converse portion of that proposition, the exponential loss is unavoidable.

The converse also has an implication for the limits of \textit{anti-concentration} results. An anti-concentration bound on a mean-$0$ random variable is an upper bound on
$\Pr(|p|<\ep (Ep^2)^{1/2})$ for some $\ep$; earlier we applied Paley-Zygmund in this way for $\ep=1/2$ but one may seek better bounds for smaller $\ep$, and 
considerable work has been done on anti-concentration for polynomials. 
However even in the most tractable case, when the underlying random variables are iid normal~\cite{CarberyW01}, bounds for polynomials of degree $\delta$ are useful only for $\Pr(|p|<(Ep^2)^{1/2} / \delta^\delta)$, i.e., for exponentially small $\ep$ as a function of $\delta$. These bounds have been extended with slight loss to hypercontractive random variables (which includes the $\pm 1$ case we study), but due to the exponential dependence on $\delta$, such bounds cannot lead to a version of Theorem~\ref{main-thm} that does not suffer the exponential in $d$. And of course since anti-concentration implies a lower bound on $E|p|$, the impossibility of a Khintchine bound sub-exponential in $\delta$ also implies that strong
anti-concentration results cannot be sought for $\delta$ more than logarithmic in $w$. Indeed the $p$ provided as an example in Prop.~\ref{hi*prop} is a constant-free Littlewood polynomial of degree $\delta$ with $2^{\delta+1}-2$
monomials, for which $\Pr(p=0)$ is 
close to $1$ (specifically $>1-2^{1-\delta}$). 

However, despite the Khintchine converse, it is an open question whether the exponential dependence on $d$ in Theorem~\ref{main-thm} is necessary.

\section*{Acknowledgments}
Part of this work was done while the authors were in residence at the Simons Institute for Theory of Computing, and part while the second author was in residence at the Israel Institute for Advanced Studies, supported by 
a EURIAS Senior Fellowship co-funded by the Marie 
Sk{\l}odowska-Curie Actions under the 7th Framework Programme. Work of the second author was also supported by United States NSF grant 1618795. Thanks to an anonymous referee for insightful comments which helped tighten our bounds.

\bibliographystyle{plain}
\bibliography{refs}
\end{document}